\documentclass{amsart}
\usepackage{hyperref}
\newtheorem{theorem}{Theorem}
\newtheorem{corollary}[theorem]{Corollary}
\newtheorem{conjecture}[theorem]{Conjecture}
\newtheorem{lemma}[theorem]{theorem}
\newtheorem{proposition}[theorem]{Proposition}
\newtheorem{problem}[theorem]{Problem}
\theoremstyle{definition}
\newtheorem{definition}[theorem]{Definition}

\theoremstyle{remark}
\newtheorem{remark}[theorem]{Remark}

\DeclareMathOperator{\Irr}{Irr}

\begin{document}

\title[A Note on Average of Roots of Unity]
 {A Note on Average of Roots of Unity}

\author[C. Panraksa]{Chatchawan Panraksa}
\address[Chatchawan Panraksa]{Science Division, Mahidol University International College\\
999 Phutthamonthon 4 Road, Salaya, Nakhonpathom, Thailand 73170}
\email{chatchawan.pan@mahidol.ac.th}

\author[P. Ruengrot]{Pornrat Ruengrot}
\address[Pornrat Ruengrot]{Science Division, Mahidol University International College\\
999 Phutthamonthon 4 Road, Salaya, Nakhonpathom, Thailand 73170}
\email{pornrat.rue@mahidol.ac.th}


\date{\today}





\begin{abstract}
 We consider the problem of characterizing all functions $f$ defined on the set of integers modulo $n$ with the property that an average of some $n$th roots of unity determined by $f$ is always an algebraic integer. Examples of such functions with this property are linear functions. We show that, when $n$ is a prime number, the converse also holds. That is, any function with this property is representable by a linear polynomial. Finally, we give an application of the main result to the problem of determining self perfect isometries for the cyclic group of prime order $p$.
\end{abstract}

\maketitle
\section{Introduction}

Let $n$ be a positive integer. Denote by $\mathbb{Z}_n=\{0,1,\ldots,n-1\}$ the ring of integers modulo $n$. Let $\omega=e^{2\pi i/n}$ be a primitive $n$th root of unity. In this work, we consider the following problem.

\begin{problem}\label{the-problem}
  Suppose $f:\mathbb{Z}_n\longrightarrow \mathbb{Z}_n$ is a function such that the average
\begin{equation}\label{eqn:integrality}
  \mu_f^{a,b} = \frac{1}{n}\sum_{x=0}^{n-1}\omega^{af(x)+bx} \quad\text{is an algebraic integer for every } a,b\in\mathbb{Z}_n.
\end{equation}
What can be said about the function $f$?
\end{problem}

When $n$ is a power of prime, such a problem is related to the problem of finding self perfect isometries (as defined in \cite{Broue1990isometries}) for cyclic $p$-groups, since the condition \eqref{eqn:integrality} is the integrality condition for perfect characters. More explanations on this relationship are given in the last section.

\section{Preliminaries}
We shall use the symbol $(=)$ to denote ordinary equality. The symbol $(\equiv)$ will be used to denote congruence $\pmod n$(i.e., equality in $\mathbb{Z}_n$).

\begin{definition}
  We say that a function $f:\mathbb{Z}_n\longrightarrow \mathbb{Z}_n$ is a \emph{polynomial function} if there exists a polynomial $F\in \mathbb{Z}_n[X]$ such that
  \[f(x)\equiv F(x)   \quad\text{ for all }x=0,1,\ldots, n-1.\]
\end{definition}

First we show that any function representable by a linear polynomial function satisfies the condition \eqref{eqn:integrality}.

\begin{lemma}\label{lem:linear-implies-integrality}
  Suppose that $f:\mathbb{Z}_n\longrightarrow \mathbb{Z}_n$ is a linear polynomial function, then $\mu_f^{a,b}$  is an algebraic integer for every $a,b\in\mathbb{Z}_n$.
\end{lemma}

\begin{proof}
  If $f$ is given by a linear polynomial, then so is $af(x)+bx$ for any $a,b$. Suppose $af(x)+bx\equiv \alpha x+\beta$ for some $\alpha, \beta\in\mathbb{Z}_n$. If $\alpha\equiv 0$, then
  \begin{align*}
     \mu_f^{a,b} &= \frac{1}{n}\sum_{x=0}^{n-1}\omega^{\beta} = \omega^{\beta},
  \end{align*} which is an algebraic integer. If $\alpha\not\equiv 0$, then $\omega^\alpha\ne 1$. Consequently,
  \begin{align*}
     \mu_f^{a,b} &= \frac{1}{n}\sum_{x=0}^{n-1}\omega^{\alpha x+\beta} = \frac{\omega^\beta}{n}\left(\frac{1-\omega^{\alpha n}}{1-\omega^\alpha}\right) = 0,
  \end{align*} which is also an algebraic integer.
\end{proof}

The next result shows that it suffices to check ~\eqref{eqn:integrality} for $a\equiv 0$ (which is trivial) or $b\equiv 0$ or $a, b$ relatively prime.

\begin{proposition}
 Let $f:\mathbb{Z}_n\longrightarrow \mathbb{Z}_n$ be a function. Suppose that $\mu_f^{a,b}$ is an algebraic integer. Then $\mu_f^{ka,kb}$  is also an algebraic integer for any $k$ relatively prime to $n$.
\end{proposition}

\begin{proof}
Since $k$ and $n$ are coprime, we can define an automorphism $\sigma_k\in\mathrm{Aut}(\mathbb{Q}(\omega)/\mathbb{Q})$ by $\sigma_k(\omega)=\omega^k$. Then
\begin{align*}
    \mu_f^{ka,kb} &= \frac{1}{n}\sum_{x=0}^{n-1}\omega^{k(af(x)+b)} = \frac{1}{n}\sum_{x=0}^{n-1}\sigma_k(\omega)^{af(x)+b} = \sigma_k(\mu_f^{a,b}).
  \end{align*}
Since $\mu_f^{a,b}$ is an algebraic integer and $\sigma_k\in\mathrm{Aut}(\mathbb{Q}(\omega)/\mathbb{Q})$, it follows that $\mu_f^{ka,kb}= \sigma_k(\mu_f^{a,b})$ is also an algebraic integer.
\end{proof}

Finally, we give a necessary and sufficient condition for the average of roots of unity to be an algebraic integer. This is a standard result in algebraic number theory.

\begin{lemma}\label{lem:average}
  Let $\omega_1,\ldots, \omega_n$ be roots of unity. Their average is an algebraic integer if and only if either $\omega_1+\cdots +\omega_n=0$ or $\omega_1=\cdots=\omega_n$.
\end{lemma}

\begin{proof}
The sufficiency is clear. Let $\mu$ denote the average of $\omega_1,\ldots, \omega_n$ and assume that it is an algebraic integer. By the triangle inequality, $|\mu|\le 1$ with equality if and only if $\omega_1=\cdots=\omega_n$. Moreover, $|\mu'|\le 1$ for all algebraic conjugates $\mu'$ of $\mu$. If not all $\omega_i$'s are equal, then $|\mu|< 1$. As a result, we also have $|\alpha|<1$, where $\alpha$ is the product of all algebraic conjugates of $\mu$. But $\alpha$ is an integer, which implies that $\alpha$ must be 0. It follows that $\mu=0$.
\end{proof}

\section{The case where $n$ is a prime}

Let $n=p$ be a prime number. Our main result is to show that any function $f:\mathbb{Z}_p\longrightarrow \mathbb{Z}_p$ satisfying \eqref{eqn:integrality} must be representable by a linear polynomial. To study a function $f:\mathbb{Z}_p\longrightarrow \mathbb{Z}_p$, it suffices to study polynomials of degree at most $p-1$. The following lemma was proved (for a general finite field) by Dickson~\cite{dickson1896analytic}.

\begin{lemma}
  For any function $f:\mathbb{Z}_p\longrightarrow \mathbb{Z}_p$, there exists a unique polynomial $F\in \mathbb{Z}_p[X]$ of degree at most $p-1$ such that
  \[f(x)\equiv F(x)  \pmod p \quad\text{ for all }x=0,1,\ldots, p-1.\]
\end{lemma}

Henceforth, we shall identify a function $f:\mathbb{Z}_p\longrightarrow \mathbb{Z}_p$ with its corresponding polynomial of degree at most $p-1$. For the proof of our main result, we will need to consider the following set.
\begin{definition}
  For a polynomial $f$ in $\mathbb{Z}_p[X]$, define
  \[W_f=\{\lambda\in\mathbb{Z}_p: x\mapsto (f(x)+\lambda x) \text{ is a permutation on } \mathbb{Z}_p\}.\]
\end{definition}

\begin{lemma}[Stothers {\cite[Theorem 2]{stothers1990permutation}}]\label{lem:suff-to-be-linear}
  Let $f$ be a polynomial in $\mathbb{Z}_p[X]$ of degree at most $p-1$. If $|W_f|>(p-3)/2$, then $\deg(f)\le 1$.
\end{lemma}

We now state our main result.

\begin{theorem}\label{thm:prime-linear}
   Let $f:\mathbb{Z}_p\longrightarrow \mathbb{Z}_p$ be a function. Suppose that the average
\begin{equation*}
  \mu_f^{1,b} = \frac{1}{p}\sum_{x=0}^{p-1}\omega^{f(x)+bx}
\end{equation*}
  is an algebraic integer for every $b\in\mathbb{Z}_p$. Then $f$ is a linear polynomial function.
\end{theorem}

\begin{proof}

By Lemma~\ref{lem:average}, we have that, for each $b$, either $f(x)+bx$ is constant modulo $p$ or $\mu_f^{1,b}=0$. Suppose there is $b_0\in\mathbb{Z}_p$ such that $f(x)+b_0x$ is constant modulo $p$. Then it is clear that $f$ is of the form $f(x)\equiv\alpha x+\beta$ for all $x\in \mathbb{Z}_p$.

 If there is no $b$ with $f(x)+bx$ constant modulo $p$, then $\sum_{x=0}^{p-1}\omega^{f(x)+bx} = 0$ for all $b$. Since the minimal polynomial of $\omega$ is $1+X+\cdots+X^{p-1}$, we must have that $x\mapsto f(x)+bx$ is a permutation modulo $p$, for all $b$. This means that the set $W_f$ has cardinality $p$. Hence, by Lemma~\ref{lem:suff-to-be-linear}, $f$ is a linear polynomial.
\end{proof}

\begin{corollary}\label{cor:prime-linear}
  Let $f:\mathbb{Z}_p\longrightarrow \mathbb{Z}_p$ be a function. Suppose that the average
\begin{equation*}
  \mu_f^{a,b} = \frac{1}{p}\sum_{x=0}^{p-1}\omega^{f(x)+bx}
\end{equation*}
  is an algebraic integer for every $a,b\in\mathbb{Z}_p$. Then $f$ is a linear polynomial function.
\end{corollary}
\begin{proof}
  Take $a=1$ and apply Theorem~\ref{thm:prime-linear}.
\end{proof}

We have seen (in Lemma \ref{lem:linear-implies-integrality}) that if a function $f:\mathbb{Z}_n\longrightarrow \mathbb{Z}_n$ is representable by a linear polynomial, then it satisfies \eqref{eqn:integrality}. Corollary~\ref{cor:prime-linear} implies that the converse also holds for $n$ prime. It is natural to make the following conjecture.

\begin{conjecture}
  Suppose $f:\mathbb{Z}_n\longrightarrow \mathbb{Z}_n$ is a function such that the average $\mu_f^{a,b}$ is an algebraic integer for every $a,b\in\mathbb{Z}_n$. Then there exist $\alpha,\beta\in\mathbb{Z}_n$ such that $f(x)\equiv \alpha x+\beta$ for all $x$.
\end{conjecture}

\begin{remark}
  It is possible that a function $f$ may have a non-linear form and still satisfy \eqref{eqn:integrality}. The conjecture says that such function should be representable by a linear polynomial in $\mathbb{Z}_n[X]$.

  For example, when $n=6$, it can be checked that $f(x)\equiv x^3+x$ satisfies \eqref{eqn:integrality}. However $f$ is representable by a linear polynomial, since $x^3+x\equiv 2x$ for all $x\in\mathbb{Z}_6$.
\end{remark}

\section{Connection to Perfect Isometries}

Corollary~\ref{cor:prime-linear} has an application in representation theory of finite groups, especially in the problem of finding self perfect isometries for the cyclic group $C_p$ of order prime $p$. For the purpose of illustrations, we will define perfect isometries specifically for this special case. Interested readers are referred to \cite{Broue1990isometries} for the definition of perfect isometries for general blocks of finite groups.

Throughout this section, let $G=C_p$. Denote by $\mathcal{R}(G)$ the free abelian group generated by $\Irr(G)$, the set of all irreducible complex characters of $G$. We will regard $\mathcal{R}(G)$ as lying in $CF(G)$\footnote{This is an inner product space with the standard inner product of group characters.}, the space of complex-valued class functions of $G$.

Let $I:\mathcal{R}(G)\longrightarrow\mathcal{R}(G)$ be a linear map. Define a generalized character $\mu_I$ of $G\times G$ by
\[\mu_I(g,h)=\sum_{\chi\in\Irr(G)}I(\chi)(g)\chi(h),\quad\text{for all } g,h\in G.\]

\begin{definition}(Cf. Definition 1.1 in \cite{Broue1990isometries})
An isometry $I:\mathcal{R}(G)\longrightarrow\mathcal{R}(G)$  is said to be a \emph{perfect isometry} if $\mu_I$ satisfies the following two conditions.
\begin{enumerate}
  \item[(i)] (Integrality) For all $g, h\in G$, the number $\mu_I(g,h)/p$ is an algebraic integer.
  \item[(ii)] (Separation) If $\mu_I(g,h)\ne 0$, then both $g$ and $h$ are the identity element or both are not.
\end{enumerate}
\end{definition}

Let $\omega=e^{2\pi i/p}$. Suppose that $G$ is generated by an element $u\in G$. For $x=0,1,\ldots, p-1$, let $\chi_x$ be the irreducible complex character of $G$ such that
\[\chi_x(u^a) =\omega^{ax},\quad a=0,1,\ldots p-1.\]
In particular, $\chi_0$ is the trivial character.

Any bijection $f$ on the set $\{0,1,\ldots, p-1\}$ gives rise to an isometry $I_f:\mathcal{R}(G)\longrightarrow\mathcal{R}(G)$ defined (on the basis) by
\[I_f(\chi_x)=\chi_{f(x)},\quad x=0,1 \ldots, p-1.\]

\begin{proposition}\label{prop-perfect-iso}
  An isometry $I_f$ is perfect if and only if $f$ is a linear bijection.
\end{proposition}
\begin{proof}
  Since every element in $G$ is of the form $u^a$ for some $a$, we have
  \[\mu_I(g,h)=\mu_I(u^a,u^b)=\sum_{x=0}^{p-1}I(\chi_x)(u^a)\chi_x(u^b)= \sum_{x=0}^{p-1} \omega^{af(x)+bx}.\]
Thus, we see that the condition in (1) is precisely the requirement that $\mu_I$ satisfies the integrality condition. This is the only condition to consider, as the separation condition is satisfied for any bijection $f$.

If $f$ is a linear bijection, then by Lemma \ref{lem:linear-implies-integrality}, the integrality condition is satisfied. Thus, $I_f$ is a perfect isometry.

Conversely, if $I_f$ is a perfect isometry, then $\mu_I(u^a,u^b)/p$ is an algebraic integer for all $a,b$. It follows from Corollary~\ref{cor:prime-linear} that $f$ must be linear.
\end{proof}

\begin{remark}
The following actions on $\Irr(G)$ are well known to give bijections on the set.
  \begin{itemize}
    \item (Multiplication by a linear character) For a fixed $\chi_\beta\in \Irr(G)$, multiplication by $\chi_\beta$ gives a bijection
    \[I_\beta:\Irr(G)\longrightarrow\Irr(G), \quad I_\beta(\chi_x)(g)=(\chi_\beta\chi_x)(g)=\chi_{\beta+x}(g).\]
    \item (Automorphism action) For a fixed $\alpha\in\{1,2,\ldots,p-1\}$, the automorphism $g\mapsto g^\alpha$ induces a bijection
    \[I_\alpha:\Irr(G)\longrightarrow\Irr(G), \quad I_\alpha(\chi_x)(g)=\chi_x(g^\alpha)=\chi_{\alpha x}(g).\]
  \end{itemize}
  Proposition \ref{prop-perfect-iso} implies that, for an isometry induced by a bijection on $\Irr(G)$ to be a perfect isometry, it must be a composition of the above two types of isometries.
\end{remark}


\end{document}